\documentclass{amsart}
\usepackage{latexsym}
\usepackage{amsthm}
\usepackage{amsmath}
\usepackage{amsfonts}
\usepackage{amssymb}
\usepackage[dvips]{graphicx}

\input xy
\xyoption{all}
\SelectTips{cm}{}

\theoremstyle{plain}
\newtheorem{theo}{Theorem}[section]
\newtheorem{lemma}[theo]{Lemma}
\newtheorem{propo}[theo]{Proposition}
\newtheorem{coro}[theo]{Corollary}

\theoremstyle{definition}
\newtheorem{defi}[theo]{Definition}

\theoremstyle{remark}
\newtheorem{rem}[theo]{Remark}

\newcommand\Sets{\mathcal{S}\mathrm{ets}}
\newcommand\Hom{\operatorname{Hom}}
\newcommand\Cell{\operatorname{Cell}}
\newcommand\cC{\mathcal {C}}
\newcommand\cE{\mathcal{E}}
\newcommand\cK{\mathcal {K}}

\newcommand\cV{\mathcal {V}}

\numberwithin{equation}{section}

\begin{document}
\title[On solid and rigid monoids in monoidal categories]{On solid and rigid
monoids\\ in monoidal categories}
\author[J.\ J.\ Guti\'errez]{Javier J.\ Guti\'errez}
\thanks{The author was supported by the NWO (SPI 61-638), the MEC-FEDER grant MTM2010-15831,
and by the Generalitat de Catalunya as a member of the team 2009~SGR~119.}
\address{Radboud Universiteit Nijmegen, Institute for Mathematics, Astrophysics, and
Particle Physics, Heyendaalseweg 135, 6525 AJ Nijmegen, The Netherlands}
\email{j.gutierrez@math.ru.nl}


\begin{abstract}
We introduce the notion of solid monoid and rigid monoid in monoidal
categories and study the formal properties of these objects in this
framework. We show that there is a one to one correspondence between solid
monoids, smashing localizations and mapping colocalizations, and prove that
rigid monoids appear as localizations of the unit of the monoidal structure.
As an application, we study solid and rigid ring spectra in the stable
homotopy category and characterize connective solid ring spectra as Moore
spectra of subrings of the rationals.
\end{abstract}

\maketitle
\section{Introduction}
A \emph{solid ring} in the sense of Bousfield--Kan~\cite{BK72} is a ring $R$
with unit whose core $cR$ is $R$ itself, where the \emph{core} 	of $R$ is
defined as
$$
cR=\{ x\in R \mid 1_R\otimes_{\mathbb{Z}} x=x\otimes_{\mathbb{Z}}1_R\}.
$$
They were also called $T$-rings in~\cite{BS77} and $\mathbb{Z}$-epimorphs
in~\cite{DS84}. Indeed, the unit $\mathbb{Z}\to R$ is an epimorphism of rings
if and only if $R$ is solid.

For example $\mathbb{Z}$, $\mathbb{Q}$ and $\mathbb{Z}/n$ are solid, but the
$p$-adic integers $\widehat{\mathbb{Z}}_p$ and $\mathbb{R}$ are not. Solid
rings are completely classified and all of them are commutative and
countable~\cite[Proposition 3.1]{BK72}. From this classification it turns
out, for example, that the only torsion-free solid rings are the subrings of
the rationals.

A \emph{rigid ring} $R$ is a ring with unit such that the evaluation at the
unit morphism
$$
\Hom_{\mathbb{Z}}(R,R)\longrightarrow R
$$
that sends $\varphi$ to $\varphi(1_R)$ is an isomorphism. The terminology of
rigid rings was first used in~\cite{CRT} in order to describe the class of
rings appearing as localizations of the circle $S^1$ in the category of
topological spaces, and as localizations of the integers in the category of
groups (see {\cite[Theorem 5.9]{CRT}}). This class of rings had been
previously studied under the name of $E$-rings~\cite{Sch73}. All of them are
commutative and they have been classified in the torsion-free finite rank
case~\cite{PV91}.

Examples of rigid rings are $\mathbb{Z}/n$, subrings of $\mathbb{Q}$, and
$\widehat{\mathbb{Z}}_p$ for any prime $p$. There are many other examples
such as all solid rings, and the products $\prod_{p\in P}\mathbb{Z}/p$ and
$\prod_{p\in P}\widehat{\mathbb{Z}}_p$, where $P$ is any set of primes.
However, there are groups such as the Pr\"ufer group $\mathbb{Z}/p^{\infty}$
or the $p$-adic field $\widehat{\mathbb{Q}}_p$ that do not admit a rigid ring
structure.

There exist rigid rings of arbitrarily large cardinality~\cite{DMV87}, and
this fact was used in~\cite{CRT}  and ~\cite{CG05} to prove that there is a
proper class of non-equivalent $f$-localizations in the category of
topological spaces and in the category of spectra, respectively.

In this paper, we define the concept of solid monoid and rigid monoid in
monoidal categories, similarly to their algebraic counterparts, and describe
the formal properties they satisfy in this framework (see
Section~\ref{sol_rig_mon}). Thus, if $(\cE,\otimes, I, \Hom_{\cE})$ is a
closed symmetric monoidal category, then a monoid $(R,\mu,\eta)$ is a
\emph{solid monoid} if $\mu$ is an isomorphism, and $R$ is a \emph{rigid
monoid} if the induced morphism
$$
\eta^*\colon \Hom_{\cE}(R,R)\longrightarrow \Hom_{\cE}(I, R)
$$
is an isomorphism in $\cE$.

Solid and rigid monoids are closely related with localization and
colocalization functors in $\cE$, and in Section~\ref{loc_coloc_mon} we
recall the basics of this theory in the setting of enriched monoidal
categories. If $R$ is a solid monoid in $\cE$, then the functors $X\mapsto
X\otimes R$ and $X\mapsto \Hom_{\cE}(R,X)$ are idempotent. We will show that
this property characterizes solid monoids. In fact, we will prove that there is
a bijection between the following classes:
\begin{itemize}
\item[{\rm (i)}] Solid monoids.
\item[{\rm (ii)}] Smashing localizations.
\item[{\rm (iii)}] Mapping colocalizations.
\end{itemize}
Here, a \emph{smashing localization functor} means a localization functor of
the form $L_AX =X\otimes A$ for a fixed $A$, and a \emph{mapping
colocalization functor} is one of the form $C_A X=\Hom_{\cE}(A, X)$ for a
fixed $A$. Moreover, we show that if $R$ is a solid monoid, then the
following categories are equivalent: the category of $R$\nobreakdash-modules,
the category of $L_R$-local objects, the category of $C_R$-colocal objects
and the category of $\eta$-local objects, where $\eta\colon I\to R$ is the
unit of $R$.

For an arbitrary localization functor $L$, we prove that the object $LI$ has
a rigid monoid structure, where $I$ denotes the unit of the monoidal structure,
and that all rigid monoids appear this way.

Finally, in Section~\ref{sol_rig_spec} we particularize our results to the
stable homotopy category and prove that if $(R,\mu,\eta)$ is a solid ring
spectrum, then the $\eta$-localization functor $L_\eta$ always commutes with
the suspension, and that the only connective solid ring spectra are Moore
spectra $MA$, with $A$ a subring of the rationals.

\bigskip
\noindent\textbf{Acknowledgements}. I would like to thank Oriol Ravent\'os
for many helpful discussions on the topic of this paper.

\section{Localizations and colocalizations in enriched categories}
\label{loc_coloc_mon}

Throughout the paper $\cE$ will denote a cocomplete closed symmetric monoidal
category, with tensor product $\otimes$, unit $I$ and internal hom
$\Hom_{\cE}(-,-)$. A~functor $F\colon \cE\rightarrow \cE'$ between symmetric
monoidal categories is called \emph{monoidal} if it is equipped with a
binatural transformation $F(-)\otimes_{\cE'} F(-)\rightarrow F(-\otimes_{\cE}
-)$ and a unit $I_{\cE'}\rightarrow F(I_{\cE})$ satisfying the usual
associativity, symmetry and unit conditions. A~symmetric monoidal functor is
called  \emph{strong} if the structure maps are isomorphisms.

Let $\cV$ be a symmetric monoidal category. Recall that a \emph{closed
symmetric mon\-oidal $\cV$-category} is a closed symmetric monoidal category
$\cE$ together with an adjunction $i\colon\cV\leftrightarrows\cE\colon r$,
where the left adjoint $i$ is strong monoidal. (Note that this automatically implies
 that $r$ is monoidal.)

Any closed symmetric monoidal $\cV$-category $\cE$ is enriched, tensored and
cotensored over $\cV$. Indeed, if $A$ is any object of $\cV$ and $X$, $Y$ are
objects of $\cE$, then we define $\Hom(X,Y)=r(\Hom_{\cE}(X, Y))$, $A\otimes
X=i(A)\otimes_{\cE} X$ and $X^A=\Hom_{\cE}(i(A), X)$. Thus, we have natural
isomorphisms
$$
\cE(A\otimes X, Y)\cong \cV(A, \Hom(X,Y))\cong\cE(X, Y^A).
$$
Conversely, any closed symmetric monoidal category $\cE$ that is enriched,
tensored and cotensored over $\cV$ is a closed symmetric monoidal
$\cV$-category, where
$$
i(A)=A\otimes I_{\cE}\;\mbox{ and }\; r=\Hom(I_{\cE}, X).
$$

Let $\cC$ be any category. A functor $L\colon\cC\rightarrow\cC$ is called
\emph{coaugmented} if it is equipped with a natural transformation $l\colon
{\rm Id}\rightarrow L$. A coaugmented functor is \emph{idempotent} if
$l_{LX}=Ll_X$ and $l_{LX}$ is an isomorphism for every $X$ in $\cC$.

Dually, a functor $C\colon\cC\rightarrow\cC$ is called \emph{augmented} if it
is equipped with a natural transformation $c\colon C\rightarrow {\rm Id}$. An
augmented functor is \emph{idempotent} if $c_{CX}=Cc_X$ and $c_{CX}$ is an
isomorphism for every $X$ in $\cC$.

\begin{defi}
A coaugmented idempotent functor~$(L,l)$ in $\cC$ is called a
\emph{localization functor}. Dually, an augmented idempotent functor~$(C,c)$
in $\cC$ is called a \emph{colocalization functor}. \label{def_loc}
\end{defi}
The closure under isomorphisms of the objects in the image of~$L$ are called
\emph{$L$\nobreakdash-local objects} and the closure under isomorphisms of the objects in
the image of~$C$ are called \emph{$C$-colocal objects}. A map $f$ is called
an \emph{$L$-local equivalence} if $L(f)$ is an isomorphism, and it is called
a \emph{$C$-colocal equivalence} if $C(f)$ is an isomorphism.

Therefore, for every $X$ in $\cC$, the morphism $l_X\colon X\to LX$ is an
$L$-local equivalence, and it has the following universal property: $l_X$ is
initial among all morphisms from $X$ to objects isomorphic to $LY$ for some
$Y$, that is, the induced map
\begin{equation}
l_X^*\colon\cC(LX, LY)\longrightarrow \cC(X, LY)
\label{univ_prop}
\end{equation}
is a bijection for every $X$ and $Y$. In fact, $L$-local objects and
$L$-local equivalences are \emph{orthogonal}, that is, if $Z$ is $L$-local
and $f\colon X\to Y$ is an $L$-local equivalence, then the induced map
\begin{equation}
f^*\colon \cC(Y,Z)\longrightarrow \cC(X,Z)
\label{orthog}
\end{equation}
is a bijection. Moreover, an object is $L$-local if and only if it is
orthogonal to all $L$-local equivalences, and a morphism is an $L$-local
equivalence if and only if it is orthogonal to all $L$-local objects.

Dually, in the case of colocalization functors, the morphism $c_X$ is
terminal among all morphism from $C$-colocal objects to $X$, and an
orthogonality relation similar to~\eqref{orthog} between $C$-colocal objects
and $C$-colocal equivalences holds.

Thus, if we want to define localization and colocalization functors in closed
monoidal categories or $\cV$-enriched categories, it makes sense to replace
the set $\cC(-,-)$ in $\eqref{univ_prop}$ and $\eqref{orthog}$, and in the
corresponding formulas for colocalizations, by the internal hom or the
$\cV$-enrichment.

\begin{defi}
Let $\cE$ be a closed symmetric monoidal $\cV$-category with associated
adjunction $i\colon \cV\leftrightarrows \cE\colon r$ and let
$\Hom(-,-)=r\Hom_{\cE}(-,-)$.
\begin{itemize}
\item[{\rm (i)}] An \emph{$(i,r)$-localization functor} is a localization
    functor $(L, l)$ in $\cE$ such that for every $L$-local equivalence
    $f\colon X\to Y$ and every $L$-local object $Z$ in $\cE$, the induced
    map
\begin{equation}
f^*\colon \Hom(Y,Z)\longrightarrow \Hom(X, Z)
\label{enriched_orthog}
\end{equation}
is an isomorphism in $\cV$.
\item[{\rm (ii)}] An \emph{$(i,r)$-colocalization functor} is a
    colocalization functor $(C, c)$ in $\cE$ such that for every
    $C$-colocal equivalence $f\colon X\to Y$ and every $C$-colocal object
    $Z$ in $\cE$, the induced map
\begin{equation}
f_*\colon \Hom(Z,X)\longrightarrow \Hom(Z, Y)
\label{co_enriched_orthog}
\end{equation}
is an isomorphism in $\cV$.
\end{itemize}
\label{def_loc_coloc}
\end{defi}
Note that in order to check if an object is $L$-local, it is enough to see
that condition $\eqref{enriched_orthog}$ holds for every $L$-equivalence of
the form $l_X\colon X\to LX$. In the same way, an object $Z$ is $C$-colocal
if and only if $\eqref{co_enriched_orthog}$ holds for every $C$-colocal
equivalence of the form $c_X\colon CX\to X$.

Let $\cE$ be a closed symmetric monoidal category and consider the strict
monoidal functor $i\colon \Sets\to \cE$ defined by $i(A)=\coprod_{a\in A} I$,
where $\Sets$ is closed monoidal with the cartesian product. This functor has
a right adjoint $r$, namely $r(X)=\cE(I, X)$. In this case, the enriched
orthogonality condition~\eqref{enriched_orthog} for an $(i,r)$-localization
reduces to condition~\eqref{orthog}.

If we view $\cE$ itself as a closed symmetric monoidal $\cE$-category, just
by taking $i$ and $r$ to be the identity functors, then an $({\rm Id}, {\rm
Id})$-localization functor in $\cE$ is a~\emph{closed localization functor}
following the terminology used in~\cite[Definition 3.3]{CGMV}. In the same
way, we define \emph{closed colocalizations} in $\cE$ as $({\rm Id}, {\rm
Id})$-colocalizations in~$\cE$.

Observe that any $(i,r)$-localization functor in $\cE$ satisfies orthogonality
condition~\eqref{orthog} for any $i$ and $r$. Indeed, by adjointness
$$
\cV(I_{\cV}, \Hom(X,Y))=\cV(I_{\cV},r\Hom_{\cE}(X,Y))\cong \cE(I_{\cE},
\Hom_{\cE}(X,Y))\cong \cE(X,Y),
$$
and thus, by applying the functor $\cV(I,-)$ to~$\eqref{enriched_orthog}$, we
get a bijection
$$
\cE(Y,Z)\longrightarrow \cE(X,Z).
$$
Similarly, one gets that any $(i,r)$-colocalization functor satisfies the
corresponding orthogonality condition at the level of sets of morphisms, for
any $i$ and $r$.

Observe that the class of $L$-local objects and the class of $C$-colocal
objects are closed under retracts. The following lemma gathers some closure
properties of local and colocal objects and equivalences with respect to the
tensor product and the internal hom.
\begin{lemma}
Let $(L,l)$ be an $(i,r)$-localization functor and $(C,c)$ be an
$(i,r)$\nobreakdash-colo\-ca\-li\-zation functor in a closed symmetric mo\-noi\-dal
$\cV$-category $\cE$. Let $A$, $B$ in $\cV$ and $X$, $Y$ in $\cE$.
\begin{itemize}
\item[{\rm (i)}] If $h$ is an $L$-local equivalence, then so is the
    tensor product $i(A)\otimes h$. If $h$ is a $C$-colocal equivalence,
    then so is $\Hom_{\cE}(i(A),h)$.
\item[{\rm (ii)}] If $f\colon i(A)\to Y$ and $g\colon X\to i(B)$ are
    $L$-local equivalences, then so is the tensor product $f\otimes g$.
\item[{\rm (iii)}] If $X$ is $L$-local, then so is $\Hom_{\cE}(i(A), X)$.
    If $X$ is $C$-colocal, then so is the tensor product $i(A)\otimes X$.
\end{itemize}
\label{closure}
\end{lemma}
\begin{proof}
We only give proofs for the statements for the case of localizations. The
dual case is proved by using similar arguments. By the enriched Yoneda lemma,
to prove (i) it is enough to check that
$$
\cV(W, r(\Hom_{\cE}(i(A)\otimes X, Z)))\longrightarrow
\cV(W, r(\Hom_{\cE}(i(A)\otimes Y, Z)))
$$
is an isomorphism for all $L$-local objects $Z$ in $\cE$ and all $W$ in
$\cV$. By adjointness and the fact that $i$ is strong monoidal we have that
\begin{multline}\notag
\cV(W, r(\Hom_{\cE}(i(A)\otimes X, Z)))\cong \cE(i(W)\otimes i(A)\otimes X, Z)\\
\cong\cE(i(W\otimes A), \Hom_{\cE}(X, Z))\cong \cV(W\otimes A, r(\Hom_{\cE}(X, Z)))\\
\cong \cV(W\otimes A, r(\Hom_{\cE}(Y, Z)))\cong \cV(W,
r(\Hom_{\cE}(i(A)\otimes Y, Z))).
\end{multline}
Part (ii) follows from (i) since $f\otimes g=(f\otimes i(B))\circ
(i(A)\otimes g)$ and the composition of $L$\nobreakdash-equivalences is an
$L$-equivalence.

To prove (iii), let $h\colon W\to Z$ be any $L$\nobreakdash-equivalence. Then
\begin{multline}\notag
\!\!\Hom(Z, \Hom_{\cE}(i(A), X))\cong r(\Hom_{\cE}(Z, \Hom_{\cE}(i(A), X)))
\cong r(\Hom_{\cE}(Z\otimes i(A), X))\\
\cong r(\Hom_{\cE}(W\otimes i(A), X)\cong \Hom(W, \Hom_{\cE}(i(A), X)),
\end{multline}
where the third isomorphism follows since $h\otimes i(A)$ is an
$L$-equivalence by (i).
\end{proof}

\begin{coro}
If $(L, l)$ is a closed localization in $\cE$, then the tensor product of an
$L$-local equivalence with any object is an $L$-local equivalence, and  if
$Z$ is $L$-local, then $\Hom_{\cE}(W,Z)$ is $L$\nobreakdash-local for any
$W$. Dually, if $(C,c)$ is a closed colocalization on~$\cE$, then the tensor
product of a $C$-colocal object with any object is again
$C$\nobreakdash-colocal, and if $h$ is a $C$-colocal equivalence, then so is
$\Hom_{\cE}(W, h)$ for any $W$. \qed
\label{coro2.5}
\end{coro}
\begin{defi}
Let $\cE$ be a closed symmetric monoidal category. A localization functor
$(L,l)$ in $\cE$ is called \emph{smashing}, if there is an object $A$ such
that $LX= X\otimes A$ for every $X$. A colocalization functor $(C,c)$ is
called \emph{mapping} if there is an object $A$ in $\cE$ such that
$CX=\Hom_{\cE}(A,X)$ for every~$X$.
\end{defi}
Note that if $(L,l)$ is smashing, then $(L,l)$ is canonically isomorphic  to
$(L, {\rm Id}\otimes l_I)$, that is, there is a canonical natural isomorphism
$\phi\colon L\to L$ such that $\phi\circ l={\rm Id}\otimes l_I$. Moreover,
$A\cong LI$ and $X\otimes LY\cong L(X\otimes Y)$ for all $X$ and $Y$ in
$\cE$. Similarly, if $C$ is mapping, then
$C(\Hom_{\cE}(X,Y))\cong\Hom_{\cE}(X, CY)$.

\begin{propo}
Let $(L,l)$ be a smashing localization functor and let $(C,c)$ be a mapping
colocalization functor in $\cE$. Then, for every closed symmetric monoidal
$\cV$-structure $(i,r)$ on $\cE$, we have the following:
\begin{itemize}
\item[{\rm (i)}] The localization functor $(L, l)$ is an
    $(i,r)$-localization.
\item[{\rm (ii)}] The colocalization functor $(C, c)$ is an
    $(i,r)$-colocalization.
\end{itemize}
\label{smashing}
\end{propo}
\begin{proof}
To prove (i) we need to check that the induced map
$$
l^*_X\colon\Hom(LX, LY)\longrightarrow \Hom(X, LY)
$$
is an isomorphism in $\cV$ for all $X$ and $Y$. By the enriched Yoneda lemma
it is enough to prove that
$$
\cV(W, \Hom(LX, LY))\longrightarrow \cV(W, \Hom(X, LY))
$$
is an isomorphism for all $W$ in $\cV$. This follows directly, since by
adjointness
\begin{multline}\notag
\cV(W, \Hom(LX, LY))\cong \cE(i(W), \Hom_{\cE}(X\otimes A, LY))\\
\cong \cE(i(W)\otimes X\otimes A, LY)) \cong \cE(L(i(W)\otimes X), LY))
\end{multline}
and using the fact that $\cE(L(i(W)\otimes X), LY)\cong \cE(i(W)\otimes X,
LY)$. Part (ii) is proved by a similar argument.
\end{proof}

We recall now an important source of examples of localization functors and
colocalization functors, namely localization with respect to morphisms and
colocalization with respect to objects.

\subsection{Localization with respect to morphisms}
\label{section_f_loc} Let $\cE$ be a closed symmetric monoidal $\cV$-category
with adjunction $(i,r)$ and let $\mathcal{L}$ be a class of morphisms in
$\cE$. Recall that $\Hom(-,-)=r\Hom_{\cE}(-,-)$.
\begin{itemize}
\item[{\rm (i)}] An object $Z$ in $\cE$ is \emph{$\mathcal{L}$-local} if
    for every $f\colon X\to Y$ in $\mathcal{L}$ the induced map
$$
f^*\colon \Hom(Y,Z)\longrightarrow \Hom(X,Z)
$$
is an isomorphism in $\cV$.
\item[{\rm (ii)}] A morphism $g\colon U\to W$ is called an
    \emph{$\mathcal{L}$-local equivalence} if the induced map
$$
g^*\colon \Hom(W,Z)\longrightarrow \Hom(U,Z)
$$
is an isomorphism in $\cV$ for every $\mathcal{L}$-local object $Z$.
\end{itemize}
\begin{defi}
An \emph{$(i,r)$-$\mathcal{L}$-localization functor} in $\cE$ is an
$(i,r)$-localization functor $(L,l)$ such that the class of $L$-local objects
coincides with the class of $\mathcal{L}$-local objects and the class of
$L$-local equivalences coincides with the class of $\mathcal{L}$-local
equivalences. We denote this localization functor by $L_{\mathcal{L}}$.
\end{defi}

Now, consider $\cE$ as an $\cE$-category and let $E$ be any object if $\cE$.
\begin{itemize}
\item[{\rm (i)}] A morphism $f$ in $\cE$ is called an
    \emph{$E$-equivalence} if $E\otimes f$ is an isomorphism
\item[{\rm (ii)}] An object $Z$ in $\cE$ is called \emph{$E$-local} if
    for every $E$-equivalence $f\colon X\to Y$, the induced map
$$
f^*\colon \Hom_{\cE}(Y, Z)\longrightarrow \Hom_{\cE}(X,Z)
$$
is an isomorphism in $\cE$.
\end{itemize}
An \emph{$E$-localization functor} is an $\mathcal{L}$-localization functor,
where the class of $\mathcal{L}$-local equivalences equals the class of
$E$-equivalences. We will denote this localization functor by $L_E$. Every
smashing localization $L$ is of this type, namely $L=L_{LI}$.

\subsection{Colocalization with respect to objects}
Now, let $\cK$ be a class of objects of $\cE$.
\label{section_coloc}
\begin{itemize}
\item[{\rm (i)}] A morphism $g\colon U\to W$ is called a
    \emph{$\mathcal{K}$-colocal equivalence} if the induced map
$$
g_*\colon \Hom(Z,U)\longrightarrow \Hom(Z,W)
$$
is an isomorphism in $\cV$ for every $Z$ in $\cK$.
\item[{\rm (ii)}] An object $Z$ in $\cE$ is \emph{$\mathcal{K}$-colocal}
    if for every $\cK$-colocal equivalence $f\colon X\to Y$ the induced
    map
$$
f_*\colon \Hom(Z,X)\longrightarrow \Hom(Z,Y)
$$
is an isomorphism in $\cV$.
\end{itemize}

\begin{defi}
An \emph{$(i,r)$-$\mathcal{K}$-colocalization functor} in $\cE$ is an
$(i,r)$-colocalization functor $(C,c)$ such that the class of $C$-colocal
objects coincides with the class of $\mathcal{K}$-local objects and the class
of $C$-colocal equivalences coincides with the class of $\mathcal{K}$-colocal
equivalences. We denote this colocalization functor by $C_{\mathcal{K}}$.
\end{defi}

\section{Solid monoids and rigid monoids}
\label{sol_rig_mon} In this section we define solid monoid and rigid monoid
in (enriched) monoidal categories as a generalization of  the algebraic
notion of solid ring~\cite{BS77, BK72, DS84} and rigid ring~\cite{CRT,
Sch73}. These special classes of monoids are closely related to localization
and colocalization functors.  In fact, as we will show, they all appear as
suitable localizations of the unit of the monoidal structure.

Throughout this section, we will implicitly assume that all localization and
colocalization functors with respect to morphisms and objects that are
mentioned exist.

Let $\cE$ be a closed symmetric monoidal category. Recall that a
\emph{monoid} $(R, \mu,\eta)$ in $\cE$ is an object $R$ equipped with two
morphisms $\mu\colon R\otimes R\rightarrow R$ and $\eta\colon I\rightarrow R$
such that the following diagrams commute:
\begin{equation}
\xymatrix{
R\otimes R\otimes R\ar[r]^-{1\otimes\mu}\ar[d]_{\mu\otimes 1} & R\otimes R
\ar[d]^{\mu} \\
R\otimes R \ar[r]_-{\mu} & R
}
\qquad
\xymatrix{
I\otimes R\ar[r]^{\eta\otimes 1}\ar[dr]_{\lambda} & R\otimes R\ar[d]^{\mu} &
R\otimes I\ar[l]_{1\otimes\eta}\ar[dl]^{\rho}\\
 & R, & }
\label{monoid}
\end{equation}
where $\lambda$ and $\rho$ are the natural isomorphisms coming from the fact
that $I$ is a left and right identity for the tensor product. A \emph{map of
monoids} between $(R,\mu,\eta)$ and $(R',\mu',\eta')$ is a map $f\colon
R\rightarrow R'$ compatible with the structure maps $\mu$, $\eta$, $\mu'$ and
$\eta'$, that is, such that $f\circ\mu=\mu'\circ({f\otimes f})$ and $f\circ
\eta=\eta'$.

\begin{defi}
A monoid $(R,\mu,\eta)$ in $\cE$ is called a \emph{solid monoid} if the
multiplication map $\mu$ is an isomorphism. A \emph{solid comonoid} is a
solid monoid in the opposite category $\cE^{\rm op}$.
\end{defi}
Note that if $R$ is a solid monoid, then by the commutativity of the second
diagram in~\eqref{monoid}, the morphisms $\eta\otimes 1$ and $1\otimes \eta$
are isomorphisms. In fact, this property characterizes solid monoids.
\begin{propo}
An object $R$ in $\cE$ is a solid monoid if and only if there exist a
morphism $\eta\colon I\rightarrow R$ such that both $\eta\otimes 1$ and
$1\otimes \eta$ are isomorphisms. \label{equiv_solid}
\end{propo}
\begin{proof}
One implication is clear from the previous remark. For the converse, assume
that we have an object $R$ in $\cE$ and a morphism $\eta\colon I\to R$ such
that $\eta\otimes 1$ and $1\otimes \eta$ are isomorphisms. First note that if
$f\colon R\to R$ is any map such that $f\circ \eta=\eta$, then $f=1$. Indeed,
we have a commutative diagram
$$
\xymatrix{
R\otimes I\ar[r]^{f\otimes 1_I} \ar[d]_{1\otimes \eta} & R\otimes I
\ar[d]^{1\otimes \eta} \\
R\otimes R \ar[r]_{f\otimes 1} & R\otimes R,
}
$$
where the vertical arrows are isomorphisms. Therefore
$$
f\otimes 1_I=(1\otimes \eta)^{-1}\circ (f\otimes 1)\circ (1\otimes \eta).
$$
But if $f\circ \eta=\eta$, then $(f\otimes 1)\circ(\eta\otimes 1)=\eta\otimes
1$. Since $\eta\otimes 1$ is an isomorphism, this forces $f\otimes 1=1\otimes
1$, and thus $f\otimes 1_I=1\otimes 1_I=1$. Consider now the following
isomorphism $g\colon R\to R$ defined as the composite
$$
\xymatrix{
R\ar[r]^{\lambda^{-1}}& I\otimes R\ar[r]^{\eta\otimes 1} &
R\otimes R\ar[r]^-{(1\otimes \eta)^{-1}} & R\otimes I \ar[r]^-{\rho} & R .
}
$$
Since $((\eta\otimes 1)\circ \lambda^{-1})\circ \eta=(\eta\otimes\eta)\circ
\lambda^{-1}=(\eta\otimes\eta)\circ \rho^{-1}=((1\otimes \eta)\circ
\rho^{-1})\circ \eta$, we have that $g\circ\eta=\eta$ and so $g=1$. This
implies that the following diagram of isomorphisms commute
$$
\xymatrix{
I\otimes R\ar[r]^{\eta\otimes 1} & R\otimes R & R\otimes I\ar[l]_{1\otimes\eta}\\
 & \ar[ul]^{\lambda^{-1}}\ar[ur]_{\rho^{-1}}R. & }
$$
Hence, there exists a unique isomorphism $\mu\colon R\otimes R\to R$
rendering the two triangles commutative. It is straightforward to check that
$(R, \mu, \eta)$ is a solid monoid.
\end{proof}

For any category $\cC$, the category of endofunctors ${\rm Fun}(\cC, \cC)$
admits a monoidal structure, where the tensor product is given by composition
and the unit is the identity functor. If $\cE$ is a closed symmetric monoidal
category, then we can define a functor $F \colon \cE\to {\rm Fun}(\cE, \cE)$
by setting $F(X)(-)=-\otimes X$ and another functor $G \colon \cE^{\rm op}\to
{\rm Fun}(\cE, \cE)$ by setting $G(X)(-)=\Hom_{\cE}(X,-)$. One can check that
both functors are faithful and reflect isomorphisms.

Moreover, $F$ preserves solid monoids and the functor $G$ sends solid monoids
to solid comonoids. Indeed, if $(R,\mu,\eta)$ is a solid monoid, then $(F(R),
{\rm Id}\otimes\eta)$ is a localization functor in $\cE$, that is, a solid monoid in
${\rm Fun}(\cE, \cE)$. In fact, $F(R)=L_R$ is the $R$-localization defined in
Section~\ref{section_f_loc}, since both functors have the same class of local
equivalences. Similarly, $(G(R), \Hom_{\cE}(\eta,{\rm Id}))$ is a colocalization
functor in $\cE$, that is, a solid comonoid in ${\rm Fun}(\cE, \cE)$. The
functor $G(R)$ is precisely the colocalization functor $C_R$ of
Section~\ref{section_coloc}.

\begin{theo}
Let $\cE$ be a closed symmetric monoidal category. Then, there is a one to
one correspondence between the following classes:
\begin{itemize}
\item[{\rm (i)}] Solid monoids.
\item[{\rm (ii)}] Smashing localization functors.
\item[{\rm (iii)}] Mapping colocalization functors.
\end{itemize}
\label{char_solid}
\end{theo}
\begin{proof}
Let $(R,\mu,\eta)$ be a solid monoid in $\cE$. Then the functor $(F(R),
{\rm Id}\otimes\eta)$ defined above is a localization functor in $\cE$ that is also
smashing and $F(R)(I)=R$. Conversely, let $(L,l)$ be a smashing localization
functor. Then, the morphisms
\begin{gather}\notag
LI\cong I\otimes LI\stackrel{l_I\otimes 1}{\longrightarrow}LI\otimes
LI=L(LI),\\ \notag LI\cong LI\otimes I\stackrel{1\otimes
l_I}{\longrightarrow}LI\otimes LI=L(LI)
\end{gather}
are isomorphisms. Thus, by Proposition~\ref{equiv_solid} $LI$ is a solid
monoid.

Similarly, the augmented functor $(G(R), \Hom_{\cE}(\eta,{\rm Id}))$ is a mapping
colocalization functor and $G(R)(I)=\Hom_{\cE}(R, I)$. And conversely, if
$(C,c)$ is a mapping colocalization functor, say $CX=\Hom_{\cE}(A,X)$, then
$c$ gives natural transformation
$$
\Hom_{\cE}(A, -)=C\longrightarrow {\rm Id}=\Hom_{\cE}(I,-).
$$
By the Yoneda lemma this corresponds to a morphism $\eta\colon I\to A$. But
this morphism satisfies that $\eta\otimes 1$ and $1\otimes\eta$ are
isomorphisms, by using the Yoneda lemma and the fact that $CCX\cong CX$. Again, by
Proposition~\ref{equiv_solid} we infer that $A$ is a solid monoid.
\end{proof}

\begin{rem}
\label{rem_equiv} The previous result tells us that for a monoid
$(R,\mu,\eta)$ in $\cE$ the following assertions are equivalent:
\begin{itemize}
\item[{\rm (i)}] $R$ is solid.
\item[{\rm (ii)}] The functor $(L,l)$ defined by $LX=X\otimes R$ and
    $l={\rm Id}\otimes\eta$ is a localization functor and $R\cong LI$.
\item[{\rm (iii)}] The functor $(C,c)$ defined by $CX=\Hom_{\cE}(R, X)$
    and $c=\Hom_{\cE}(\eta, {\rm Id})$ is a colocalization functor.
\end{itemize}
\end{rem}

\begin{propo}
Let $\cE$ be a closed symmetric monoidal category. If $L$ is a smashing
localization functor in $\cE$, then  $L$ is equivalent to the closed
localization functor $L_{l_I}$, where $l_I\colon I\to LI$ denotes the
localization of the unit. \label{prop_smashing_mor}
\end{propo}
\begin{proof}
By Theorem~\ref{char_solid} and Remark~\ref{rem_equiv}, if $L$ is smashing,
then $L=L_R$, where $R\cong LI$. The morphism $l_I$ is an $R$-equivalence,
since $R$ is a solid monoid. Thus, every $L$-local equivalence is an
$R$-equivalence. Conversely, let $f\colon X\to Y$ be an
$R$\nobreakdash-equivalence, $Z$ any $L$-local object and consider the
following commutative diagram
$$
\xymatrix{
\Hom_{\cE}(Y,Z)\ar[r] & \Hom_{\cE}(X, Z)  \\
\Hom_{\cE}(Y\otimes LI,Z)\ar[r]\ar[u] & \Hom_{\cE}(X\otimes LI, Z) \ar[u]
}
$$
The vertical arrows are isomorphisms since tensoring $l_I$ with any object is
an $l_I$\nobreakdash-local equivalence by Corollary~\ref{coro2.5}. The bottom
arrow is also an isomorphism since $X\to Y$ is an $R$-equivalence, hence
$X\otimes LI\to Y\otimes LI$ is an isomorphism. Therefore, the top arrow is
also an isomorphism.
\end{proof}

\subsection{Modules over solid monoids}
Recall that if $(R,\mu,\eta)$ is a monoid in a closed symmetric monoidal
category $\cE$, then an $R$-module consists of an object $X$ together with a
morphism $m\colon R\otimes X\to X$ such that the following diagrams commute:
$$
\xymatrix{
R\otimes R\otimes X \ar[r]^-{\mu\otimes 1} \ar[d]_{1\otimes m} &
R\otimes X \ar[d]^m \\
R\otimes X \ar[r]^-m & X}\qquad
\xymatrix{
I\otimes X\ar[r]^{\eta\otimes 1}\ar[dr]_{\lambda} & R\otimes X\ar[d]^m &
X\otimes I\ar[l]_{1\otimes\eta}\ar[dl]^{\rho}\\
 &X. & }
$$
Morphisms of $R$-modules are those compatible with the module structure. We
will denote by $R$-mod the category of $R$-modules.
\begin{theo}
Let $\cE$ be a closed monoidal category and let $(R,\mu, \eta)$ be a solid
monoid in $\cE$. We denote by {\rm $L_R$-loc} the full subcategory of  $L_R$-local
objects and by {\rm $C_R$\nobreakdash-coloc} the full subcategory of
$C_R$\nobreakdash-colocal objects. Then there is an equivalence of categories
$$
L_R{\rm{\text-loc}}\cong R{\rm{\text-mod}}\cong C_R{\rm{\text-coloc}}.
$$
\label{equiv_cat}
\end{theo}
\begin{proof}
If $X$ is an $R$-module, then $X$ is a retract of $X\otimes R$ and also a
retract of $\Hom_{\cE}(R,X)$. But, by Remark~\ref{rem_equiv}, $LX=X\otimes R$
and $CX=\Hom_{\cE}(R, X)$ define a localization and a colocalization functor,
namely $L_R$ and $C_R$, respectively. Thus, $X$ is a retract of an
$L_R$-local object and also a retract of a $C_R$-colocal object. Since local
and colocal objects are closed under retracts, the natural maps
$$
X\cong X\otimes I\stackrel{1\otimes\eta}{\longrightarrow} X\otimes R \;\mbox{ and }\;
\Hom_{\cE}(R, X)\stackrel{\eta^*}{\longrightarrow}\Hom_{\cE}(I,X)\cong X
$$
are isomorphisms.
\end{proof}
Note that this theorem together with Proposition~\ref{prop_smashing_mor}
imply that if $(R,\mu,\eta)$ is a solid monoid, then the categories $R$-mod
and the full subcategory of $\eta$-local objects are also equivalent, where
$\eta\colon I\to R$ is the unit of the ring.

\subsection{Rigid monoids in enriched categories}
Let $\cE$ be a closed symmetric monoi\-dal $\cV$-category, with associated
adjunction $i\colon \cV\leftrightarrows \cE\colon r$.
\begin{defi}
A monoid $(R,\mu,\eta)$ in $\cE$ is called an \emph{$(i,r)$-rigid monoid} if
the induced morphism
$$
\eta^*\colon \Hom(R,R)\longrightarrow \Hom(I,R)
$$
is an isomorphism in $\cV$.
\end{defi}

Observe that if $\cE$ is a $\cV$-category, then solid monoids in $\cE$ are
defined using the tensor product in $\cE$, while rigid monoids in $\cE$ are
defined in terms of the $\cV$\nobreakdash-enrichment, which depends on the
adjunction $(i,r)$. However, as we will prove, the class of solid monoids is
contained in the class of rigid monoids.

\begin{propo}
Let $\cE$ be a closed symmetric monoidal $\cV$-category, with associated
adjunction $i\colon \cV\leftrightarrows \cE\colon r$. If $(R,\mu,\eta)$ is a
solid monoid in $\cE$, then $R$ is an $(i,r)$-rigid monoid.
\label{solid_is_rigid}
\end{propo}
\begin{proof}
By Theorem~\ref{char_solid} the functor $LX=X\otimes R$ is a smashing
localization functor in $\cE$. Now, Proposition~\ref{smashing} implies that
$L$ is an $(i,r)$-localization functor. Hence, for every $L$-local $Z$ we
have an isomorphism
$$
\eta_I^*\colon \Hom(LI, Z)\longrightarrow \Hom(I, Z).
$$
In particular, taking $Z=LI\cong R$ we obtain that $R$ is an $(i,r)$-rigid monoid.
\end{proof}

\begin{propo}
Let $(L,\eta)$ be an $(i,r)$-localization functor in a closed symmetric
monoidal $\cV$-category $\cE$. If $(i(A),\mu,
\eta)$ is a (commutative) monoid in $\cE$ and $Li(A)\cong i(B)$ for some $A$
and $B$ in $\cV$, then $Li(A)$ admits a unique (commutative) monoid structure
such that the localization map $\eta_{i(A)}\colon i(A)\to Li(A)$ is a
morphism of monoids. \label{loc_rings}
\end{propo}
\begin{proof}
The unit map $\overline{\eta}$ of $Li(A)$ is $\eta_{i(A)}\circ \eta$. The
product map $\overline{\mu}$ is defined by using the universal property of
the localization and adjointness. Indeed, we have natural bijections
\begin{multline}\notag
\cE(i(A)\otimes i(A), Li(A))\cong \cE(i(A), \Hom_{\cE}(i(A), Li(A)))\\
\cong\cV(A, r\Hom_{\cE}(i(A), Li(A)))\cong \cV(A, r\Hom_{\cE}(Li(A), Li(A))) \\
\cong \cE(i(A)\otimes i(B), Li(A))\cong \cV(B, r\Hom_{\cE}(i(A), Li(A))) \\
\cong \cE(i(B), \Hom_{\cE}(Li(A), Li(A)))\cong \cE(Li(A)\otimes Li(A),
Li(A)).
\end{multline}
Hence, the product $\mu$ extends to a unique map $\overline{\mu}\colon
Li(A)\otimes Li(A)\to Li(A)$ rendering commutative the diagram
$$
\xymatrix{
i(A)\otimes i(A)\ar[d]_{\eta_{i(A)}\otimes \eta_{i(A)}} \ar[r]^-{\mu} &
i(A)\ar[d]^{\eta_{i(A)}} \\
Li(A)\otimes Li(A) \ar@{.>}[r]_-{\overline{\mu}} & Li(A).
}
$$
The associativity of $\overline{\mu}$ and its compatibility with
$\overline{\eta}$ follows from the commutativity of the diagrams for $\mu$
and $\eta$ and the universal property of $L$ (using Lemma~\ref{closure}).

In the same way one can prove that $Li(A)$ is commutative when $i(A)$ is
commutative.
\end{proof}

\begin{theo}
Let $\cE$ be a closed monoidal $\cV$-category with associated adjunction
$(i,r)$ and let $(L,\eta)$ be an $(i,r)$-localization functor in $\cE$. If
$LI\cong i(B)$ for some $B$ in $\cV$, then $LI$ is an $(i,r)$-rigid monoid in
$\cE$. In fact, all $(i,r)$-rigid monoids appear as $LI$, for some
(i,r)-localization functor $L$ in $\cE$. \label{thm3.9}
\end{theo}
\begin{proof}
Since $i(I_{\cV})=I$ and $LI\cong i(B)$ by assumption, we may apply
Proposition~\ref{loc_rings} to infer that $LI$ is again a monoid in $\cE$.
Moreover, $LI$ is $L$-local and $\eta_I\colon I\to LI$ is an $L$-equivalence.
Thus
$$
\Hom(LI, LI)\longrightarrow \Hom(I, LI)
$$
is an isomorphism, and hence $LI$ is an $(i,r)$-rigid monoid.

Conversely, suppose that $(R,\mu,\eta)$ is an $(i,r)$-rigid monoid and let
$\eta\colon I\to R$ be its unit. Then the $(i,r)$-$\eta$-localization (see
Section~\ref{section_f_loc}) satisfies $L_\eta I\cong R$. Indeed, $\eta$ is
an $\eta$-equivalence and $R$ is $\eta$-local, since
$$
\eta^*\colon \Hom(R, R)\longrightarrow \Hom(I, R)
$$
is an isomorphism, because $R$ is an $(i,r)$-rigid monoid.
\end{proof}

\begin{coro}
All $(i,r)$-rigid monoids of the form $i(B)$ for some $B$ in $\cV$ are
commutative.
\end{coro}
\begin{proof}
If $i(B)$ is $(i,r)$-rigid, then we know by Theorem~\ref{thm3.9} that
$i(B)=LI$ for some $(i,r)$-localization functor $L$. But now,
Proposition~\ref{loc_rings} implies that $LI$ is a commutative monoid since
$I=i(I_{\cV})$.
\end{proof}
\section{Solid ring spectra and rigid ring spectra}
\label{sol_rig_spec} We will apply now the results of the previous section to
the stable homotopy category of spectra $\mathcal{S}{\rm p}$; see~\cite{Ada74,
HPS97}. This is a triangulated category equipped with a compatible closed
symmetric monoidal structure. We denote by $\wedge$ the smash product, by $S$
the sphere spectrum, by $\Sigma$ the suspension operator, and by $F(-,-)$ the
internal function spectrum. We write $[X,Y]$ for the abelian group of
morphisms between two spectra $X$ and $Y$ in $\mathcal{S}{\rm p}$ and we say that a
spectrum $X$ is \emph{connective} if $\pi_k(X)=[\Sigma^k S, X]=0$ for $k<0$.
We denote the full subcategory of connective spectra by ${\rm
conn}(\mathcal{S}{\rm p})$. There is a \emph{connective cover functor} $(-)^c$ that
assigns to every spectrum $X$ a connective spectrum $X^c$ and a natural map
$$
c_X\colon X^c\longrightarrow X
$$
such that $\pi_k(c_X)$ is an isomorphism for all $k\ge 0$.

We will be interested in defining localization and colocalization functors in
$\mathcal{S}{\rm p}$ coming from the following two situations (see
Definition~\ref{def_loc_coloc} for notation):
\begin{itemize}
\item[{\rm (I)}] $\cV=\cE=\mathcal{S}{\rm p}$ and $i=r={\rm Id}$. In this case,
    $\Hom(-,-)=F(-,-)$
\item[{\rm (II)}] $\cV={\rm conn}(\mathcal{S}{\rm p})$, $\cE=\mathcal{S}{\rm p}$, $i$
    is the inclusion and $r$ is the connective cover functor. In this
    case, $\Hom(-,-)=F^c(-,-)$ is the connective cover of the function
    spectrum.
\end{itemize}
We will refer to the localizations and colocalizations in (I) as
\emph{stable}. (These were called \emph{closed} in
Section~\ref{loc_coloc_mon}.) This terminology reflects the fact that the
localizations and colocalizations coming from (I) always commute with the
suspension operator ---that is, they are exact or triangulated functors in
$\mathcal{S}{\rm p}$--- while the ones coming from (II) do not necessarily have this
property.

Examples of stable localizations are given by \emph{homological
localizations}; see~\cite{Bou79}. Using the notation of
Section~\ref{section_f_loc}, these correspond to the $E$-localizations
functors. Given a spectrum~$E$, a homological localization functor with
respect to $E$ is a localization functor $L_E$ on $\mathcal{S}{\rm p}$ that turns
$E$\nobreakdash-homology equivalences into homotopy equivalences in a
universal way. Recall that each spectrum $E$ gives rise to a homology theory
defined as $E_k(X)=\pi_k(E\wedge X)$ for every spectrum $X$ and every
$k\in\mathbb{Z}$. A  map of spectra $f\colon X\to Y$ is an
\emph{$E$\nobreakdash-equivalence} if the induced map
$$
f_*\colon E_k(X)\longrightarrow E_k(Y)
$$
is an isomorphism for all $k\in\mathbb{Z}$. A spectrum $Z$ is
\emph{$E$\nobreakdash-local} if each $E$\nobreakdash-equivalence $X\to Y$
induces a homotopy equivalence $[Y,Z]\cong [X,Z]$. Given a homological
localization functor $L_E$, there is an associated stable colocalization
functor $A_E$ constructed by taking the fiber of the localization map. Thus,
for every spectrum $X$ we have an exact triangle
$$
A_E X\longrightarrow X\longrightarrow L_E X\longrightarrow \Sigma A_E X.
$$
The functor $A_E$ is called the \emph{$E_*$-acyclization} functor
in~\cite{Bou79}. Miller's \emph{finite localizations}~\cite{Mill} are
smashing, and therefore homological localizations. Other smashing
localizations include localizations at sets of primes, and homological
localization with respect to the spectrum $K$ of (complex) $K$-theory or the
Johnson--Wilson spectrum $E(n)$ for any $n$.

As examples of localizations and colocalizations of type (II), we have $k$th
\emph{Postnikov sections} $P_{\Sigma^k S}$ and $k$th \emph{connective covers}
$\Cell_{\Sigma^{k+1} S}$. In the notation of Sections~\ref{section_f_loc}
and~\ref{section_coloc} (with $i$ the inclusion and $r$ the connective
cover), they correspond to the functors $P_{\Sigma^kS}=L_{\Sigma^k S\to 0}$
and  $\Cell_{\Sigma^{k+1} S}=C_{\Sigma^{k+1} S}$, respectively. For any spectrum $X$
we have that
$$
\pi_n(\Cell_{\Sigma^{k+1} S} X)=\left\{
\begin{array}{cr}
0 & \mbox{if $n\le k$}, \\
\pi_n(X) & \mbox{if $n> k$},
\end{array}
\right.
\quad
\pi_n(P_{\Sigma^{k}S} X)\cong\left\{
\begin{array}{cr}
0 & \mbox{if $n\ge k$}, \\
\pi_n(X) & \mbox{if $n<k$}.
\end{array}
\right.
$$
Neither $\Cell_{\Sigma^{k+1} S}$ nor $P_{\Sigma^k S}$ commute with suspension.
Note that the connective cover functor $(-)^c$ is precisely $\Cell_{S}$. More
generally, \emph{$f$-localizations}~\cite{CG05} and
\emph{$E$\nobreakdash-cellu\-la\-ri\-zations}~\cite{Gut12} in~$\mathcal{S}{\rm p}$ are also functors
of type (II).

A ring spectrum is \emph{solid} if the multiplication map is a homotopy
equivalence, and a ring spectrum $R$ is called a \emph{rigid ring spectrum}
if it is an $(i,r)$-rigid monoid with~$i$ and $r$ as in (II), that is, if the
connective cover of the evaluation map $F^c(R, R)\to R^c$ is a homotopy
equivalence. A ring spectrum is a \emph{stable rigid ring spectrum} if it is
an $(i,r)$-rigid monoid with $i=r={\rm Id}$ as in (I), that is, if the
evaluation map $F(R,R)\to R$ is a homotopy equivalence.

Given an abelian group $A$, we denote by $HA$ its corresponding
\emph{Eilenberg--Mac Lane spectrum} and by $MA$ its corresponding
\emph{Moore spectrum}. The former is characterized by the property that
$\pi_k(HA)=A$ if $k=0$ and it is zero if $k\ne 0$, and the latter is
characterized by the property that it is connective, $(H\mathbb{Z})_k(MA)=0$
if $k\ne 0$,  and $(H\mathbb{Z})_0(MA)=\pi_0(MA)=A$.

By Proposition~\ref{solid_is_rigid}, every solid ring spectrum is a rigid
ring spectrum and a stable rigid ring spectrum, but the converse does not
hold in general. For instance, the ring spectrum $H\widehat{\mathbb{Z}}_p$,
where $\widehat{\mathbb{Z}}_p$ are the $p$-adic integers is rigid but not
solid (neither stable rigid). If it were solid, then
$H\widehat{\mathbb{Z}}_p\wedge H\widehat{\mathbb{Z}}_p\cong
H\widehat{\mathbb{Z}}_p$ and this would imply that
$\widehat{\mathbb{Z}}_p\otimes \widehat{\mathbb{Z}}_p\cong
\widehat{\mathbb{Z}}_p$. However
$F^c(H\widehat{\mathbb{Z}}_p,H\widehat{\mathbb{Z}}_p)\cong
H\widehat{\mathbb{Z}}_p$, since $[S, H\widehat{\mathbb{Z}}_p]\cong
[H\widehat{\mathbb{Z}}_p, H\widehat{\mathbb{Z}}_p]$ and $[\Sigma^k
H\widehat{\mathbb{Z}}_p, H\widehat{\mathbb{Z}}_p]=0$ for all $k\ge 1$.

Applying Theorems~\ref{char_solid} and~\ref{thm3.9} to the category
$\mathcal{S}{\rm p}$ readily implies

\begin{theo}
Let $L$ be a localization functor in $\mathcal{S}{\rm p}$ and let $S$ be the sphere spectrum. Then we have the
following:
\begin{itemize}
\item[{\rm (i)}] If $L$ is smashing (hence stable and homological), then
    $LS$ is a solid ring spectrum, and all solid ring spectra appear as
    smashing localizations of the sphere spectrum.
\item[{\rm (ii)}] If $L$ is any localization functor and $LS$ is
    connective, then $LS$ is a rigid ring spectrum and all rigid ring
    spectra appear as localizations of the sphere spectrum.
\item[{\rm (iii)}] If $L$ is a stable localization functor, then $LS$ is
    a stable rigid ring spectrum and all stable rigid ring spectra appear
    as stable localizations of the sphere spectrum. \qed
\end{itemize}
\label{char_spec}
\end{theo}
More explicitly, if $R$ is a solid ring spectrum, then $R\cong L_R S$; and if
$R$ is a rigid ring spectrum, then $R\cong L_\eta S$, where $\eta\colon S\to
R$ is the unit of the ring spectrum~$R$. If $R$ is a stable rigid ring
spectrum, then $R\cong L_{\Sigma^*\eta}S$, where $\Sigma^*\eta=\{\Sigma^k
\eta\mid k\in\mathbb{Z}\}$

For any solid ring spectrum $R$, the colocalization functor $C_R$ is
precisely stable $R$-cellularization $\Cell_R$, and the localization functor
$L_R$ is homological localization with respect to $R$. If we denote by
$L_R\mathcal{S}{\rm p}$ the full subcategory of $R$-local spectra and by
$\Cell_{R}\mathcal{S}{\rm p}$ the full subcategory of $R$-cellular spectra, then
Theorem~\ref{equiv_cat} gives

\begin{propo}
If $R$ is a solid ring spectrum, e.g., $R=H\mathbb{Q}$, $L_K S$ or
$L_{E(n)}S$, then there is an equivalence of categories
$L_R{\mathcal{S}{\rm p}}\cong R{\rm{\text-mod}}\cong \Cell_R{\mathcal{S}{\rm p}}$.
\qed
 \label{prop_equiv_cat}
\end{propo}

Observe that Proposition~\ref{prop_smashing_mor} implies that any of the
categories of Proposition~\ref{prop_equiv_cat} is also equivalent to the
category of $\eta$-local objects, where $\eta\colon S\to R$ is the unit
of~$R$. This has the following consequence:
\begin{coro}
If $(R,\mu,\eta)$ is a solid ring spectrum, then the $R$-homological
localization functor $L_R$ is equivalent to the stable localization $L_\eta$.
\label{solid_map}
\end{coro}
\begin{rem}
Proposition~\ref{prop_equiv_cat} is a particular instance of a more general
result in the context of stable model categories~\cite[Theorem 2.7]{GS} and,
in fact, the above equivalence is induced by a Quillen equivalence at 
the level of the corresponding model categories.
Also, as proved in \cite[Corollary 4.14]{BR12}, Corollary~\ref{solid_map}
also holds at the level of the localized model structures, that is, the left
Bousfield localizations $L_R$ and $L_{\eta}$ coincide.
\end{rem}

The following result (see \cite[Theorems 5.12 and 5.14]{CG05}) relates
$f$-localizations of the integral Eilenberg-Mac\,Lane spectrum $H\mathbb{Z}$
with algebraic rigid rings (that is, rigid monoids in the category of abelian
groups).

\begin{theo}
Let $L_f$ be any $f$-localization functor in $\mathcal{S}{\rm p}$. Then, $L_f
H\mathbb{Z}\cong HA$ for some rigid ring $A$ and all (algebraic) rigid rings
appear this way. If $L_f$ is smashing, then $A$ is a subring of the
rationals. \qed
\end{theo}

Moreover,  the only connective solid ring spectra are Moore spectra of
subrings of the rationals.

\begin{theo}
Let $L$ be any localization functor in $\mathcal{S}{\rm p}$ and assume that $LS$ is
connective. Then $LS$ is a solid ring spectrum if and only if $LS\cong MA$,
where $A$ is a subring of the rationals. \label{thm_new}
\end{theo}
\begin{proof}
Suppose that $LS$ is a solid ring. Then, by Theorem~\ref{char_spec}(i), we
have that $LS\cong L'S$, where $L'$ is a smashing localization functor (in fact, $L'=L_{LS}$). 
Now, it follows from \cite[Theorem 5.14]{CG05} that
$(H\mathbb{Z})_k(L'S)=0$ if $k\ne 0$ and $(H\mathbb{Z})_0(L'S)\cong
R$ a subring of the rationals. Thus, $LS\cong MA$, since it is connective.

The converse holds since if $R$ is a subring of the rationals, then the
multiplication map $MA\wedge MA\to MA$ is an isomorphism.
\end{proof}

\begin{coro}
If $R$ is a connective solid ring spectrum, then $R\cong MA$ for some subring
of the rationals $A$. \qed
\end{coro}

\end{document}